\documentclass[a4paper,12pt]{article}
\usepackage{amsmath,amssymb,amsthm,enumerate}
\usepackage[T2A]{fontenc}  
\usepackage[cp866]{inputenc}
\usepackage[russian,english]{babel}
\begin{document}

\author{\bfseries\large K.~O.~Besov\thanks{Steklov Mathematical Institute
of Russian Academy of Sciences, ul.~Gubkina~8, Moscow, 119991 Russia.\protect\\
E-mail: kbesov@mi.ras.ru}}

\title{On Balder's Existence Theorem\\ for Infinite-Horizon Optimal Control Problems\thanks
{This work is supported by the Russian Science Foundation under
grant~14-50-00005.}}

\makeatletter

\def\@fnsymbol#1{\ensuremath{\ifcase#1\or*\or**\or***\or \,d\or \,e\else\@ctrerr\fi}}
\makeatother

\maketitle

\begin{flushleft}
\end{flushleft}

{\small\leftskip=10mm\rightskip=\leftskip\noindent
Balder's well-known existence theorem (1983) for infinite-horizon optimal
control problems is extended to the case when the integral functional is
understood as an improper integral. Simultaneously, the condition of strong
uniform integrability (over all admissible controls and trajectories) of
the positive part $\max\{f_0,0\}$ of the utility function (integrand)~$f_0$
is relaxed to the requirement that the integrals of~$f_0$ over intervals
$[T,T']$ be uniformly bounded from above by a function $\omega(T,T')$ such
that $\omega(T,T')\to 0$ as $T,T'\to\infty$. This requirement was proposed
by A.V.~Dmitruk and N.V.~Kuz'kina (2005); however, the proof in the present
paper does not follow their scheme but is instead derived in a rather
simple way from the auxiliary results of Balder himself. An illustrative
example is also given.

\medskip MSC2010: 49J15, 49J45

}\smallskip

\newtheorem{teA}{Theorem} \def\theteA{\Alph{teA}}
\newtheorem{teo}{Theorem}
\newtheorem{lem}{Lemma}
\newtheorem{prp}{Proposition}
\theoremstyle{definition}
\newtheorem{rem}{Remark}
\newtheorem{exa}{Example}

\let\eps\varepsilon
\let\al\alpha \let\om\omega \let\Om\Omega \let\la\lambda
\def\tom{\wt\om}
\let\wt\widetilde \let\wh\widehat
\let\qq\qquad
\def\R{\mathbb R} \def\N{\mathbb N}
\def\cA{\mathcal A} \def\cB{\mathcal B} \def\cL{\mathcal L}
\def\cU{\mathcal U}
\def\({\left(}
\def\){\right)}
\def\[{\left[}
\def\]{\right]}
\let\To\rightrightarrows
\def\ACloc{\mathrm{AC}_\textup{loc}}
\def\1{\kern1pt}

\vskip 8mm

One of the most general and well-known results on the existence of
solutions to infinite-horizon optimal control problems was proved by
Balder~\cite{Ba}. Almost all conditions of his theorem are local in time
(i.e., they must hold only at each separate instant of time or on each
finite time interval) and ensure the existence of solutions to similar
problems on finite time intervals. The only condition that regulates the
behavior of the system at infinity is the requirement of strong uniform
integrability of the positive part of the integrand in the objective
functional over all admissible controls and corresponding trajectories.
Later several authors achieved some progress in weakening this condition.

The present paper also contributes to this direction. As an alternative to
Balder's uniform integrability, we use the condition of ``uniform
boundedness of pieces of the objective functional'' proposed by Dmitruk and
Kuz'kina \cite{DK}. Note that they considered a significantly narrower
class of optimal control problems, while for the general case only a scheme
was outlined (without statement of particular results that can be obtained
by following this scheme\footnote
{
 The absence of exact statements to which one could refer when solving
 particular optimal control problems was one of the reasons for writing
 the present note.
}).
So the present paper is in a sense a logical completion of the paper
\cite{DK}. However, we do not follow the scheme proposed in~\cite{DK} but
rather show that the result can be derived from those of Balder himself
\cite{Ba81,Ba} in a fairly simple way.

Recently Bogusz \cite{Bo} also obtained an existence theorem in the case
when the integral functional is understood as an improper integral.
However, one of the hypotheses in her theorem is the existence of a locally
integrable function $\la\colon\R_+\to\nobreak\R$ that has a finite improper
integral $\int_0^\infty\la(t)\,dt$ and bounds from above (from below in the
case of minimization problem) the integrand in the objective functional for
all admissible controls and corresponding trajectories. Such a condition is
essentially stronger (although formally this is not so) than the strong
uniform integrability, because subtracting (adding in the case of
minimization problem) the function~$\la$ from (to) the integrand reduces
the problem to the one with negative (positive) integrand in the objective
functional.

Some results on the existence of optimal solutions under conditions of
different kind and/or in different statements of the problem were obtained
in \cite{As16,Ly}.

Note that existence theorems are an inherent part of the method for solving
optimal control problems based on applying necessary optimality conditions
(see, e.g., \cite{ABKr,As15,As17,Be14,Be15}). Therefore, it is important to
have an existence theorem under hypotheses maximally close to those under
which necessary optimality conditions are valid. At present it is the
condition of ``uniform boundedness of pieces of the objective functional''
that is often required for necessary optimality conditions to be valid
(see, e.g.,~\cite{ABKr,Be14}).

\medskip

Let us proceed to the statement of the problem and formulate the conditions
under which we will study it.

The main object of our study is the optimal control problem
\begin{gather}\label{eq1}
  I(x,u):=\int_0^\infty f_0(t,x(t),u(t))\,dt\to{\max},\\[3pt] \label{eq2}
  \dot x(t)=f(t,x(t),u(t))\qq\text{for a.e. }\, t\in\R_+:=[0,+\infty),\\[5pt] \label{eq3}
  x(t)\in A(t),\qq u(t)\in U(t,x(t))\qq\text{for a.e. }\, t\in\R_+,
\end{gather}
for which the following conditions hold (where $m,n\in\N$ are fixed
dimensions of the control and state vectors, respectively):
\begin{enumerate}[(i)]\itemsep=-1pt
\item $A\colon\R_+\To\R^n$ is a set-valued map with $(\cL\times\cB^n)$-measurable\footnote
{
 That is, lying in the $\sigma$-algebra generated in $\R_+\times\R^n$
 by the Cartesian products of Lebesgue measurable subsets in~$\R_+$
 and Borel subsets in~$\R^n$.
} graph~$\cA$;
\item $U\colon\cA\To\R^m$ is a set-valued map with $(\cL\times\cB^{n+m})$-measurable graph~$\cU$;
\item the functions $f\colon\cU\to\R^n$ and $f_0\colon\cU\to\R\cup\{-\infty\}$ are $(\cL\times\cB^{n+m})$-measurable\footnote
{
 That is, the preimages of Borel sets are $(\cL\times\cB^{n+m})$-measurable.
}.
\end{enumerate}
The set~$\Om$ of \textit{admissible pairs}~$(x,u)$ consists by definition
of pairs of vector functions $x$,~$u$ such that $x\in\ACloc^n(\R_+)$,
$u\colon\R_+\to\R^m$ is a Lebesgue measurable function and conditions
\eqref{eq2} and~\eqref{eq3} hold. Here $\ACloc^n(\R_+)$ is the space of
locally absolutely continuous (i.e., absolutely continuous on each finite
interval) functions $x\colon\R_+\to\R^n$ with the topology indicated
in~\cite{Ba}.

The integral in \eqref{eq1} is understood in~\cite{Ba} in the following
sense:
\begin{equation}\label{+-}
  \int_0^\infty g(t)\,dt:=\int_0^\infty g^+(t)\,dt-\int_0^\infty g^-(t)\,dt,\qq
  \text{where}\quad g^\pm:=\max\{\pm g,0\},
\end{equation}
with the convention\footnote
{
 Here and below, without further mention, we reformulate all the results of \cite{Ba,DK}
 obtained for minimization problems as applied to similar maximization problems.
}
that $(+\infty)-(+\infty)=-\infty$. Thus, the value of the functional
\eqref{+-} (equal to a finite number or~$\pm\infty$) is defined for any
admissible pair.

Fix an $\al\in\R$ and put $\Om_\al:=\{(x,u)\in\Om\mid I(x,u)\ge\al\}$. The
existence of a solution to problem~\eqref{eq1}--\eqref{eq3} is proved
in~\cite{Ba} under the following assumptions:
\begin{enumerate}[(i)]\itemsep=-1pt\setcounter{enumi}{3}
\item the function $f(t,\cdot\1,\cdot\1)$ is continuous on
$\cU(t):=\{(\chi,\upsilon)\in\R^n\times\R^m\mid (t,\chi,\upsilon)\in\cU\}$
for every~$t\in\R_+$;
\item the function $f_0(t,\cdot\1,\cdot\1)$ is upper semicontinuous on~$\cU(t)$ for every $t\in\R_+$;
\item the sets $A(t)$ and $\cU(t)$ are closed for every $t\in\R_+$;
\item the set $\{x(0)\mid (x,u)\in\Om_\al\}$ is bounded;
\item for every $T>0$, the set of functions $F^T_\al:=\{f(\1\cdot\1,x(\cdot),u(\cdot))|_{[0,T]}\mid (x,u)\in\Om_\al\}$
is uniformly integrable on the interval~$[0,T]$, i.e.,
$\inf_{c>0}\sup_{g\in F^T_\al} \int_{C^T_{g,c}}\mathopen|g(t)|\,dt=0$,
where $C^T_{g,c}=\{t\in[0,T]\mid |g(t)|>c\}$;
\item the set $Q(t,\chi):=\{(z^0,z)\in\R\times\R^n\mid z^0\le f_0(t,\chi,\upsilon),\ z=f(t,\chi,\upsilon),\ \upsilon\in U(t,\chi)\}$
is convex for all $(t,\chi)\in\cA$;
\item $Q(t,\chi)=\bigcap_{\delta>0}\mathrm{cl}\bigl(\bigcup_{\chi'\in A(t)\cap B_\delta(\chi)} Q(t,\chi')\bigr)$,
where $B_\delta(\chi)$ is the ball of radius~$\delta$ centered at~$\chi$;
\item the set of functions $F_{0,\al}^+:=\{f_0^+(\1\cdot\1,x(\cdot),u(\cdot))\mid (x,u)\in\Om_\al\}$
is strongly uniformly integrable on~$\R_+$, i.e., $\inf_{h\in L_1(\R_+)}
\sup_{g\in F_{0,\al}^+} \int_{C_{g,h}}\mathopen|g(t)|\,dt=0$, where
$C_{g,h}:=\{t\in\R_+\mid |g(t)|>\nobreak h(t)\}$.
\end{enumerate}

\begin{teA}[{\cite[Theorem~3.6]{Ba}}]\label{tA}
If there is an $\al\in\R$ such that $\Om_\al\ne\varnothing$ and
conditions~\textup{(i)\nobreakdash--(xi)} hold, then in problem
\eqref{eq1}--\eqref{eq3} there exists an admissible pair $(x_*,u_*)\in\Om$
such that\/ $I(x_*,u_*)=\sup_{(x,u)\in\Om}I(x,u)$.
\end{teA}

As was already mentioned, the only condition in Theorem~\ref{tA} that
regulates the behavior of system~\eqref{eq1}--\eqref{eq3} at infinity is
condition~(xi). At the same time, in many optimal economic growth problems
it seems more natural to define the value of the objective functional not
in the sense of~\eqref{+-} but rather in the limit sense
\begin{equation}\label{J}
  J(x,u):=\lim_{T\to+\infty}\int_0^T f_0(t,x(t),u(t))\,dt
\end{equation}
provided that the limit exists (see, e.g., \cite{ABKr,Bo}). We will also
follow this definition, in which case problem \eqref{eq1}--\eqref{eq3} is
replaced by the problem
\begin{equation}\label{eq1'}
  J(x,u)\to{\max}
\end{equation}
subject to conditions \eqref{eq2} and \eqref{eq3}.

\begin{rem}\label{r1}
It is clear that if the value of the functional $I(x,u)$ is finite for an
admissible pair~$(x,u)$, then $J(x,u)=I(x,u)$.
\end{rem}

As noticed in \cite{DK}, instead of condition~(xi) one can consider the
condition
\begin{enumerate}
\item[(xii)] $\varlimsup_{T\to+\infty}\sup_{T'>T}\sup_{(x,u)\in\Om}\int_T^{T'}\!f_0(t,x(t),u(t))\,dt\le 0$.
\end{enumerate}
It is easy to see that for admissible pairs in~$\Om_\al$ condition~(xii) is
weaker\footnote
{\label{footn}%
 On the whole it would be incorrect to say that condition~(xii) is weaker than condition~(xi),
 because condition~(xii) is considered for the set~$\Om$ while condition~(xi) is considered
 only for the subset $\Om_\al\subset\Om$. Therefore, formally none of the conditions implies the other.
}
than condition~(xi). Indeed,
\begin{multline*}
  \varlimsup_{T\to+\infty}\,\sup_{T'>T}\,\sup_{(x,u)\in\Om_\al}\int_T^{T'}f_0(t,x(t),u(t))\,dt
  \le\varlimsup_{T\to+\infty}\,\sup_{T'>T}\,\sup_{g\in F_{0,\al}^+}\int_T^{T'}g(t)\,dt \\
  \le\inf_{h\in L_1(\R_+)}\,\varlimsup_{T\to+\infty}\,\sup_{T'>T}\,\sup_{g\in F_{0,\al}^+}
    \left(\int_{[T,T']\cap C_{g,h}}g(t)\,dt+\int_{[T,T']\setminus C_{g,h}}h(t)\,dt\right) \\
  \le \inf_{h\in L_1(\R_+)}\,\varlimsup_{T\to+\infty}\,\sup_{T'>T}\,\sup_{g\in F_{0,\al}^+}\int_{C_{g,h}}g(t)\,dt+0
  = \inf_{h\in L_1(\R_+)}\,\sup_{g\in F_{0,\al}^+}\int_{C_{g,h}}g(t)\,dt.
\end{multline*}

However, below we will still need a local version of condition~(xi),
namely,
\begin{enumerate}
\item[(xi$'$)] for every $T>0$, the set of functions $F_0^{T,+}:=\{f_0^+(\1\cdot\1,x(\cdot),u(\cdot))|_{[0,T]}\mid (x,u)\in\Om\}$
is uniformly integrable on~$[0,T]$.
\end{enumerate}
(In \cite{DK}, due to the continuity and compact-valuedness of the
functions and set-valued maps considered there, condition~(xi$'$) holds
automatically.)

Let us make the following important observation.

\begin{prp}\label{a1}
Under condition~\textup{(xi$'$),} condition~\textup{(xii)} is equivalent to
each of the following conditions\textup:
\begin{enumerate}
\item[\upshape(xii$'$)] there is a continuous function $\om\colon\R_+^2\to\R_+$ such that
$\om(T,T')\to 0$ as $T,T'\to\infty$ and
$$
  \sup_{(x,u)\in\Om}\int_T^{T'}f_0(t,x(t),u(t))\,dt\le\om(T,T')\qq\forall T,T'\colon\ T'>T\ge0;
$$
\item[\upshape(xii$''$)] there is a continuous function $\tom\colon\R_+\to\R_+$ such that
$\tom(T)\to 0$ as $T\to\infty$ and
$$
  \sup_{T'>T}\,\sup_{(x,u)\in\Om}\int_T^{T'}f_0(t,x(t),u(t))\,dt\le\tom(T)\qq\forall T\ge0.
$$
\end{enumerate}
\end{prp}

\begin{proof}
Clearly, condition~(xii$''$) implies condition~(xii$'$) (it suffices to
take $\om(T,T'):=\tom(T)$) and condition~(xii$'$) implies condition~(xii)
(for $\varlimsup_{T\to\infty}\sup_{T'>T}$ is the same as
$\varlimsup_{T,T'\to\infty,\;T'>T}$, while the latter does not exceed
$\varlimsup_{T,T'\to\infty}$). Let us show that condition~(xii) implies
condition~(xii$''$). Put
$$
  \wh\om(T):=\left(\sup_{T'>T}\sup_{(x,u)\in\Om}\int_T^{T'}\!f_0(t,x(t),u(t))\,dt\right)^{\!+},\qq T\ge 0.
$$
Due to condition~(xii) we have $\lim_{T\to\infty}\wh\om(T)=0$. Therefore,
there is a~$T_1$ such that $\wh\om(T)\le 1$ for $T\ge T_1$. Let us show
that this function is bounded for all~$T\ge 0$. For $T<T_1$ we have
$$
  \wh\om(T)\le \sup_{(x,u)\in\Om}\int_0^{T_1}\!f_0^+(t,x(t),u(t))\,dt+\wh\om(T_1)
  \le \inf_{c>0}\,\sup_{g\in F_0^{T_1,+}}\left(\int_{C^{T_1}_{g,c}} g(t)\,dt+cT_1\right)+1.
$$
Due to condition~(xi$'$) there is a constant $c_1>0$ such that
$$
  \sup_{g\in F_0^{T_1,+}}\int_{C^{T_1}_{g,c_1}} g(t)\,dt\le 1.
$$
Then $\wh\om(T)\le c_1T_1+2$ for all $T\ge 0$.

Put $\wh\om_1(T):=\sup_{T'\ge T}\wh\om(T')$ for $T\ge 0$. Then $\wh\om_1$
is a bounded monotonically nonincreasing function on~$\R_+$ that tends to
zero as $T\to\infty$.

Finally, put $\tom(T):=\int_{T-1}^T\wh\om_1(t^+)\,dt$ (recall that
$t^+=\max\{t,0\}$). It is clear that $\tom$ is a continuous function
on~$\R_+$ that satisfies all requirements formulated in
condition~(xii$''$).
\end{proof}

One of the important corollaries to condition~(xii) is that the value of
the functional~$J(\cdot,\cdot)$ is defined on any admissible pair. For
completeness, we will give a slightly shorter proof of this fact than
in~\cite{DK}.

\begin{prp}
Under conditions~\textup{(xi$'$)} and~\textup{(xii),} the value of the
functional~$J(x,u)$ is defined for every admissible pair $(x,u)\in\Om$ and
is equal to either a finite number or~$-\infty$.
\end{prp}

\begin{proof}
The existence of a limit in \eqref{J} follows from the estimate
\begin{multline*}
  \varlimsup_{T\to+\infty}\int_0^T f_0(t,x(t),u(t))\,dt
  = \varliminf_{T_1\to+\infty}\,\varlimsup_{T\to+\infty}\left(\int_0^{T_1}+\int_{T_1}^T\right) f_0(t,x(t),u(t))\,dt \\
  \le \varliminf_{T_1\to+\infty}\int_0^{T_1}f_0(t,x(t),u(t))\,dt
    +\varlimsup_{T_1\to+\infty}\,\sup_{T>T_1}\int_{T_1}^T f_0(t,x(t),u(t))\,dt \\
  \le \varliminf_{T\to+\infty}\int_0^T f_0(t,x(t),u(t))\,dt,
\end{multline*}
where we have used condition~(xii) at the last step. At the same time, the
limit does not exceed $\tom(0)$ for a continuous function
$\tom\colon\R_+\to\R_+$.
\end{proof}

Now we formulate our main result. By analogy with the set~$\Om_\al$, we
introduce the set $\wt\Om_\al:=\{(x,u)\in\Om\mid J(x,u)\ge\al\}$ for
$\al\in\R$.

\begin{teo}\label{t1}
If there is an $\al\in\R$ such that $\wt\Om_\al\ne\varnothing$ and
conditions~\textup{(i)--(x), (xi$'$)} and~\textup{(xii)}
\textup(or~\textup{(xii$'$),} or~\textup{(xii$''$))} hold with~$\Om_\al$
replaced by~$\wt\Om_\al$, then in problem \eqref{eq1'}, \eqref{eq2},
\eqref{eq3} there exists an admissible pair $(x_*,u_*)\in\Om$ such that
$J(x_*,u_*)=\sup_{(x,u)\in\Om}J(x,u)$.
\end{teo}

The main role in the proof is played by another result of Balder.

\begin{teA}[{\cite[Theorem~3.2]{Ba}}]\label{tB}
Suppose conditions~\textup{(i)--(vi), (ix)} and~\textup{(x)} hold. Suppose
also that $\{(x_k,u_k)\}_{k=1}^\infty$ is a sequence in~$\Om$ such that the
sequence $\{x_k\}_{k=1}^\infty$ converges weakly\/\footnote
{
 For a definition of the weak convergence in $\ACloc^n(\R_+)$, see \cite{Ba}.
}
to a function $x_0\in\ACloc^n(\R_+)$ and the set of functions
$\{f_0^+(\1\cdot\1,x_k(\cdot),u_k(\cdot))\}_{k=1}^\infty$ is strongly
uniformly integrable on~$\R_+$. Then there exists a Lebesgue measurable
function $u_*\colon\R_+\to\R^m$ such that $(x_0,u_*)\in\Om$ and
$$
  I(x_0,u_*)\ge\varlimsup_{k\to\infty}I(x_k,u_k).
$$
\end{teA}

\begin{proof}[Proof of Theorem \ref{t1}]
Let $\{(x_k,u_k)\}_{k=1}^\infty$ be a maximizing sequence for
$J(\1\cdot\1,\cdot\1)$ from~$\wt\Om_\al$. Due to conditions~(vii),~(viii)
(with~$\Om_\al$ replaced by~$\wt\Om_\al$) and Theorem~2.1 from \cite{Ba},
the sequence $\{x_k\}_{k=1}^\infty$ contains a subsequence that converges
weakly to some function $x_0\in\ACloc^n(\R_+)$. Let us pass to this
subsequence, denoting it again by~$\{(x_k,u_k)\}_{k=1}^\infty$.

For $N\in\N$ introduce the function
$$
  f_0^N(t,\chi,\upsilon):=\begin{cases} f_0(t,\chi,\upsilon),& t\in[N-1,N),\ (t,\chi,\upsilon)\in\cU,\\
    0,& t\in\R_+\setminus[N-1,N),\ (t,\chi,\upsilon)\in\cU,
  \end{cases}
$$
and consider problem \eqref{eq1}--\eqref{eq3} with $f_0^N$ instead
of~$f_0$. Denote the corresponding functional (in which the integral is
actually taken over the interval~$[N-1,N)$) by~$I_N$. Assume first that for
every $N\in\N$ all hypotheses of Theorem \ref{tB} hold for this truncated
problem (with the objective functional~$I_N$) and for our sequence
$\{(x_k,u_k)\}_{k=1}^\infty$. Then there exists a Lebesgue measurable
function $u_{N*}\colon\R_+\to\R^m$ such that $(x_0,u_{N*})\in\Om$ and
$$
  I_N(x_0,u_{N*})\ge\varlimsup_{k\to\infty}I_N(x_k,u_k).
$$
Put $u_*(t):=u_{N*}(t)$ for $t\in[N-1,N)$, $N\in\N$. Clearly,
$(x_0,u_*)\in\Om$ and
\begin{multline*}
  J(x_0,u_*)=\lim_{K\to\infty}\sum_{N=1}^{K}I_N(x_0,u_*)
  =\lim_{K\to\infty}\sum_{N=1}^{K}I_N(x_0,u_{N*})
  \ge\lim_{K\to\infty}\sum_{N=1}^{K}\varlimsup_{k\to\infty} I_N(x_k,u_k)\ge\\
  \ge\lim_{K\to\infty}\,\varlimsup_{k\to\infty} \sum_{N=1}^{K}I_N(x_k,u_k)
  \ge\lim_{K\to\infty}\,\varlimsup_{k\to\infty}\bigl(J(x_k,u_k)-\tom(K)\bigr)
  =\varlimsup_{k\to\infty}J(x_k,u_k),
\end{multline*}
where $\tom$ is a function from condition~(xii$''$).

Thus, $(x_0,u_*)$ is the required admissible pair. It remains to explain
why the conclusion of Theorem~\ref{tB} holds for the truncated problem with
functional~$I_N$. Among all hypotheses of Theorem~\ref{tB}, only
conditions~(ix) and~(x) need to be checked for $t\notin[N-1,N)$. The
validity of condition~(ix) follows from the fact that a projection of a
convex set is convex. However, the validity of condition~(x) is fairly
difficult to prove (moreover, we are even not sure that it takes place).

To overcome this difficulty, we proceed as follows. Note that in the above
reasoning the values of~$u_{N*}(t)$ are used only for $t\in[N-1,N)$.
Therefore, we can arbitrarily vary the
sequence~$\{(x_k,u_k)\}_{k=1}^\infty$ and the parameters of problem
\eqref{eq1}--\eqref{eq3} outside the interval~$[N-1,N)$. In particular, we
can set $f(t,\cdot\1,\cdot\1)=0$, $A(t)=\R^n$, $U(t,\cdot\1)=\{0\}$ and
$u_k(t)=0$ for $t\notin[N-1,N)$, as well as $x_k(t)=x_k(N-1)$ for $0\le
t<N-1$ and $x_k(t)=x_k(N)$ for $t\ge N$. For the problem thus modified
(still with the functional~$I_N$) the validity of all hypotheses of
Theorem~\ref{tB} is undoubtable, and we get the desired function~$u_{N*}$
on~$[N-1,N)$.
\end{proof}

\begin{rem}
From the formal point of view, Theorem \ref{t1} cannot be said to
strengthen Theorem~\ref{tA} not only for reasons explained in
footnote~\ref{footn} but also in view of the following important remark.
Theorems~\ref{t1} and~\ref{tA} deal with problems in which the objective
functionals are defined differently. In particular, it may happen that for
the same parameters of the problem, an optimal solution exists in one
problem and does not exist in the other, or that optimal solutions exist in
both problems but are different. Nevertheless, the hypothesis in
Theorem~\ref{t1} concerning the behavior of the control system at infinity
seems to be essentially weaker than that in Theorem~\ref{tA}. As an
illustration, we give the following example.
\end{rem}

\begin{exa}\label{ex1}
Consider the problem
\begin{gather}\label{e1}
  \int_0^\infty \frac{u(t)}{t+1}\,dt\to{\max},\\[3pt] \label{e2}
  \dot x(t)=u(t)\qq\text{for a.e. }\, t\in\R_,\\[5pt] \label{e3}
  x(t)\in[-t,t]\cap[-1,1],\qq u(t)\in [-1,1]\qq\text{for a.e. }\, t\in\R_+.
\end{gather}
It is clear that $x(0)=0$ and the absolute value of the integrand
in~\eqref{e1} is bounded by~$1/(t+1)$ for every admissible pair~$(x,u)$.
All local conditions~(i)--(x) and~(xi$'$) are satisfied. Let us show that
condition~(xii$''$) also holds:
\begin{multline}\label{xii''}
  \int_T^{T'}\frac{u(t)}{t+1}\,dt
  =\int_T^{T'}\frac{\dot x(t)}{t+1}\,dt
  =\frac{x(T')}{T'+1}-\frac{x(T)}{T+1}+\int_T^{T'}\frac{x(t)}{(t+1)^2}\,dt \\
  \le \frac{1}{T'+1}+\frac{1}{T+1}+\int_T^{T'}\frac{dt}{(t+1)^2}
  =\frac{2}{T+1}\qq \forall T>0.
\end{multline}

Thus, if we consider the functional \eqref{e1} as an improper integral,
i.e., in the sense of~\eqref{J}, then we can apply Theorem~\ref{t1}, which
guarantees the existence of an optimal solution.

This optimal solution can easily be found explicitly. Indeed, since
$$
  \lim_{T\to\infty}\int_0^{T}\frac{u(t)}{t+1}\,dt
  =\lim_{T\to\infty}\left(\frac{x(T)}{T+1}+\int_0^T\frac{x(t)}{(t+1)^2}\,dt\right)
  =\lim_{T\to\infty}\int_0^T\frac{x(t)}{(t+1)^2}\,dt,
$$
it suffices to maximize~$x(t)$ at every~$t$ (which is possible here), i.e.,
set $u_*(t)=1$ for $t<1$ and $u_*(t)=0$ for $t\ge 1$. The corresponding
optimal trajectory is $x_*(t)=\min\{t,1\}$.

Since the integrand is positive, by Remark \ref{r1} this solution is also
optimal in the case when the objective functional is understood in the
sense of~\eqref{+-}. Let us show that nevertheless Theorem~\ref{tA} is
inapplicable in this case for any~$\al$ (except for~$\al$ equal to the
exact value~$\al_*$ ($=\ln2$) of the functional on the optimal solution,
but in this case the theorem is almost worthless, because the
set~$\Om_{\al_*}$ consists of a single admissible pair). The reason is that
condition~(xi) of strong uniform integrability does not hold for
$\al<\al_*$. Let us demonstrate this.

Consider first an admissible pair with $u(t)=\cos t$, i.e., the pair
$$
  u_0(t)=\cos t, \qq
  x_0(t)=\sin t,\qq t\ge 0.
$$
Then
\begin{equation}\label{int=8}
  \int_0^\infty \left(\frac{u_0(t)}{t+1}\right)^{\!+}dt
  =\int_0^\infty \frac{\max\{\cos t,0\}}{t+1}\,dt=+\infty;
\end{equation}
i.e., no family of functions containing the integrand in~\eqref{int=8} can
be strongly uniformly integrable.

To show that condition~(xi) is violated even for admissible pairs for which
the value of the functional (in any sense) is close to the optimal value,
it suffices to construct such an admissible pair from pieces:
\begin{itemize}
\item first, on a sufficiently large interval $[0,T_1]$ with $T_1=\pi/2+2\pi k$, $k\in\N$,
use the optimal control~$u_*$ and follow the optimal trajectory~$x_*$;
\item second, on a sufficiently large interval $[T_1,T_2]$, use the
control~$u_0$ and follow the trajectory~$x_0$ (since $x_0(T_1)=1=x_*(T_1)$,
we can switch from one trajectory to the other);
\item for $t>T_2$, use the control~$u=0$.
\end{itemize}
Due to the vanishing control on the last interval, the value of the
functional (in any sense) on such an admissible pair is finite. By virtue
of estimate~\eqref{xii''} (note that replacing $u$ with~$-u$ changes the
trajectory~$x$ to~$-x$, so estimate \eqref{xii''} also holds for the
absolute value of the integral on the left-hand side), the value of the
functional (in any sense) on such a pair differs from the optimal value by
at most~$4/(T_1+1)$. Choosing a sufficiently large~$T_2$ (depending
on~$T_1$), we can make the integral analogous to~\eqref{int=8} as large as
desired. This means that condition~(xi) of strong uniform integrability
does not hold for~$\Om_\al$ for any~$\al<\al_*$.
\end{exa}

\begin{rem}
A similar example can be constructed without state constraints. For
example, it suffices to replace $u(t)$ with $u(t)(1-x(t)^2)$ in~\eqref{e1}
and \eqref{e2} and introduce the initial condition $x(0)=0$ in \eqref{e3}
instead of the state constraint.
\end{rem}

\begin{rem}
For the problem considered in Example~\ref{ex1}, the existence result from
\cite[Theorem~7.9]{Bo} is also inapplicable, because it requires that there
should be a locally integrable function $\la\colon\R_+\to\nobreak\R$ with
finite improper integral $\int_0^\infty\la(t)\,dt$ that would majorize the
integrand in the objective functional for all admissible pairs
in~$\wt\Om_\al$. It is clear that there is no such a function in our
problem for $\al<\al_*$ (while for $\al=\al_*$, as mentioned above, the
set~$\wt\Om_\al$ consists of a single pair and the theorem becomes almost
worthless).
\end{rem}


\begin{thebibliography}{}

\bibitem{As15}
{S.~M.~Aseev}, ``Adjoint variables and intertemporal prices in infinite-horizon optimal control problems,''
Tr. Mat. Inst. im.~V.A.~Steklova, Ross. Akad. Nauk~\textbf{290} (2015),
239--253 [Engl. transl.: Proc. Steklov Inst. Math.~\textbf{290} (2015),
223--237].

\bibitem{As16}
{S.~M.~Aseev}, ``Existence of an optimal control in infinite-horizon problems with unbounded set
of control constraints,'' Trudy Inst. Mat. Mekh. UrO RAN~\textbf{22}:2
(2016), 18--27.

\bibitem{As17}
{S.~M.~Aseev}, ``Optimization of dynamics of a control system in the presence of risk factors,''
Trudy Inst. Mat. Mekh. UrO RAN~\textbf{23}:1 (2017), 27--42.


\bibitem{ABKr}
{S.~M.~Aseev, K.~O.~Besov, and A.~V.~Kryazhimskiy},
``Infinite-horizon optimal control problems in economics,'' Usp. Mat. Nauk,
\textbf{67}:2 (2012), 3--64 [Engl. transl.: Russ. Math. Surv.~\textbf{67}
(2012), 195--253].


\bibitem{Ba81}
{E.J. Balder}, ``Lower semicontinuity of integral functionals with
nonconvex integrands by relaxation-compactification,'' SIAM J. Control
Optim.~\textbf{19}:4 (1981), 533--542.

\bibitem{Ba}
{E.J. Balder}, ``An existence result for optimal economic growth
problems,'' J.~Math. Anal. Appl.~\textbf{95} (1983), pp.~195--213.

\bibitem{Be14}
{K.~O.~Besov}, ``On necessary optimality conditions for infinite-horizon economic growth
problems with locally unbounded instantaneous utility function,'' Tr. Mat.
Inst. im.~V.A.~Steklova, Ross. Akad. Nauk~\textbf{284} (2014), 56--88
[Engl. transl.: Proc. Steklov Inst. Math.~\textbf{284} (2014), 50--80].

\bibitem{Be15}
{K.~O.~Besov}, ``Problem of optimal endogenous growth with exhaustible resources and possibility
of a technological jump,'' Tr. Mat. Inst. im.~V.A.~Steklova, Ross. Akad.
Nauk~\textbf{291} (2015), 56--68 [Engl. transl.: Proc. Steklov Inst.
Math.~\textbf{291} (2015), 49--60].

\bibitem{Bo}
{D.~Bogusz}, ``On the existence of a classical optimal solution and of
an almost strongly optimal solution for an infinite-horizon control
problem,'' J.~Optim Theory Appl.~\textbf{156} (2013), 650--682.

\bibitem{DK}
{A.~V.~Dmitruk and N.~V.~Kuz'kina}, ``Existence theorem in the optimal control
problem on an infinite time interval,'' Mat. Zametki~\textbf{78}:4 (2005),
503--518 [Engl. transl.: Math. Notes~\textbf{78} (2005), 466--480];
``Letter to the editor,'' Mat. Zametki~\textbf{80}:2 (2006), 320 [Engl.
transl.: Math. Notes~\textbf{80} (2006), 309].



\bibitem{Ly}
{V.~Lykina}, ``An existence theorem for a class of infinite horizon
optimal control problems,'' J.~Optim. Theory Appl.~\textbf{169} (2016),
50--73.

\label{bibl}
\end{thebibliography}
\end{document}